\documentclass[11pt]{article}

\usepackage[margin=1.1in]{geometry}
\usepackage{amsmath,amssymb,amsthm}
\usepackage{array}
\usepackage{hyperref}
\hypersetup{colorlinks=true,linkcolor=blue,urlcolor=blue,citecolor=blue}

\theoremstyle{plain}
\newtheorem{theorem}{Theorem}[section]
\newtheorem{proposition}[theorem]{Proposition}
\newtheorem{lemma}[theorem]{Lemma}
\newtheorem{conjecture}[theorem]{Conjecture}

\theoremstyle{definition}
\newtheorem{remark}[theorem]{Remark}

\newcommand{\C}{\mathbb{C}}
\newcommand{\R}{\mathbb{R}}
\newcommand{\Q}{\mathbb{Q}}
\newcommand{\kk}{\Bbbk}
\newcommand{\JC}{\mathrm{JC}}
\newcommand{\HC}{\mathrm{HC}}
\newcommand{\Hess}{\operatorname{Hess}}
\newcommand{\Jac}{\operatorname{Jac}}
\newcommand{\adj}{\operatorname{adj}}
\providecommand{\coloneqq}{\mathrel{\mathop:}=}

\title{\bf A five-variable counterexample to the Hessian\\
conjecture, and the low-dimensional status of the\\
Jacobian and Hessian conjectures\thanks{The authors are listed in
alphabetical order.}}
\author{Guowu Meng\\[2pt]
\small Department of Mathematics, Hong Kong University of Science and
Technology\\[-2pt]
\small \texttt{mameng@ust.hk}
\and
Liang Yang\\[2pt]
\small Department of Mathematics, Sichuan University, Chengdu 610064,
PR China\\[-2pt]
\small \texttt{malyang@scu.edu.cn}}
\date{Version 4 (arXiv v2) --- July 26, 2026}

\begin{document}
\maketitle

\begin{abstract}
We exhibit an explicit integer polynomial in five variables, of total
degree $14$ and with constant Hessian determinant $128$, whose gradient
is not injective. Consequently its formal Legendre transform is not a
polynomial, and the Hessian conjecture $\HC_5$ is false. The counterexample is obtained from the
six-variable doubling of Alp\"oge's 2026 Jacobian counterexample by a
one-variable \emph{Schur descent}---a partial Legendre transform in a
single variable. Combined with de~Bondt's theorem that
$\HC_n$ holds for $n\le3$, with the elementary doubling and
stabilization bridges relating the Jacobian conjectures $\JC_n$ to the
Hessian conjectures $\HC_n$, and with Alp\"oge's refutation of $\JC_3$,
this decides the Hessian conjecture in every dimension except $n=4$:
$\HC_n$ is true for $n\le3$, false for $n\ge5$, and open only at $n=4$.
Exactly two statements of the two families remain unsettled, $\JC_2$
and $\HC_4$, linked by $\HC_4 \Rightarrow \JC_2$. Along the way we
record, as a warm-up, an explicit six-variable counterexample to
$\HC_6$ with constant Hessian determinant $-4$ and non-injective
gradient. This note adds the five-variable counterexample to, and
updates the status recorded in, the first author's earlier educational
preprint \cite{MengRG2026}.
\end{abstract}



\section{Introduction}

Throughout, $\kk$ is a field of characteristic zero (for instance $\Q$,
$\R$, or $\C$). For a polynomial
map $F=(F_1,\dots,F_n)\colon \kk^n\to\kk^n$ we write
$\Jac F=\det\bigl(\partial F_i/\partial x_j\bigr)$.

\begin{conjecture}[$\JC_n$, Keller \cite{Keller1939}]
If $F\colon \kk^n\to\kk^n$ is a polynomial map with
$\Jac F\in\kk^{\times}$, then $F$ is invertible, with polynomial
inverse.
\end{conjecture}

\noindent A polynomial map $F$ with $\Jac F\in\kk^{\times}$ (that is,
with $\Jac F$ a nonzero constant) is called a \emph{Keller map}; thus the
conjecture asserts that every Keller map has a polynomial inverse.

The conjecture, open since 1939 in every dimension $n\ge2$ (see
\cite{BCW1982} for background and reductions and \cite{vandenEssen2000}
for a survey), was refuted in July 2026:
Alp\"oge publicly announced \cite{Alpoge2026} an explicit polynomial map
$F\colon\C^3\to\C^3$ with $\Jac F=-2$ possessing a fiber of three points
(the announcement credits the AI system \emph{Fable} with the
construction). With $u=1+x_1x_2$, the map is
\begin{equation}\label{eq:Fmap}
F(x_1,x_2,x_3)=
\begin{pmatrix}
u^3x_3+x_2^2\,u\,(4+3x_1x_2)\\[2pt]
x_2+3x_1u^2x_3+3x_1x_2^2(4+3x_1x_2)\\[2pt]
2x_1-3x_1^2x_2-x_1^3x_3
\end{pmatrix},
\end{equation}
of component degrees $7,6,4$. One checks that $\Jac F\equiv-2$ (a single
$3\times3$ determinant; verified independently in exact arithmetic in
Appendix~\ref{app:computations}), while
\begin{equation}\label{eq:collision}
F\Bigl(0,0,-\tfrac14\Bigr)
=F\Bigl(1,-\tfrac32,\tfrac{13}2\Bigr)
=F\Bigl(-1,\tfrac32,\tfrac{13}2\Bigr)
=\Bigl(-\tfrac14,\,0,\,0\Bigr).
\end{equation}
Since three distinct rational points share the same image, $\JC_3$
fails over every field of characteristic zero. We rely on the
announcement \cite{Alpoge2026} only for the map $F$ and its origin;
the identities \eqref{eq:Fmap}--\eqref{eq:collision}, and every
consequence drawn from them below, are explicit. Although Alp\"oge's map
was found by AI and the Schur descent below with AI assistance, the
resulting refutations of $\JC_3$ and $\HC_5$ can be checked entirely by
hand---and, independently, in exact rational arithmetic
(Appendix~\ref{app:computations})---so the mathematical content of this
paper is independent of \cite{Alpoge2026}.

This paper has two purposes. First and foremost, we construct an
explicit counterexample to the Hessian conjecture $\HC_5$
(Section~\ref{sec:psi5}): a five-variable integer polynomial of total degree
$14$ with constant nonzero Hessian determinant and non-injective
gradient, obtained from Alp\"oge's map \eqref{eq:Fmap} by doubling
followed by a one-variable Schur descent. Second, we record the resulting status of
the two families $\JC_n$ and the closely related Hessian conjectures
$\HC_n$ introduced in \cite{Meng2006}. Together with de~Bondt's theorem
that $\HC_n$ holds for $n\le3$ \cite{deBondt2012} and two standard
bridges, the new counterexample decides $\HC_n$ in every dimension but
$n=4$, collapsing the two dimension-indexed families of open problems to
exactly two low-dimensional statements, $\JC_2$ and $\HC_4$
(Theorem~\ref{thm:status}). An earlier, single-author version recording
only the status portion, intended for educational purposes, is
available as \cite{MengRG2026}; the present paper adds the $\HC_5$
counterexample and updates the status to include the truth of $\HC_2$
and $\HC_3$, established by de~Bondt \cite{deBondt2012} and not recorded
there.

\section{The Hessian conjecture and the two bridges}

Let $\varphi\in\kk[x_1,\dots,x_n]$ satisfy
$\det\Hess\varphi\in\kk^{\times}$, where
$\Hess\varphi=\bigl(\partial^2\varphi/\partial x_i\partial
x_j\bigr)$. Then $\Hess\varphi(0)$ is invertible, so the gradient map
$\nabla\varphi$ admits a formal inverse at $\nabla\varphi(0)$ (after
harmless affine normalization), and one defines the \emph{formal
Legendre transform} $\varphi^{L}$ of $\varphi$ by the classical formula
$\varphi^{L}(p)=\langle p,x(p)\rangle-\varphi(x(p))$, where $x(p)$ is
the formal inverse of $p=\nabla\varphi(x)$; equivalently,
$\nabla\varphi^{L}=(\nabla\varphi)^{-1}$ as formal maps. This is the
classical change of variables $x\leftrightarrow p=\nabla\varphi(x)$
familiar from convex analysis and mechanics, carried out formally near
the base point. In general
$\varphi^{L}$ is a formal power series; see \cite{Meng2006}, where a
tree formula for its coefficients is given.

\begin{conjecture}[$\HC_n$, \cite{Meng2006}]
If $\varphi\in\kk[x_1,\dots,x_n]$ has
$\det\Hess\varphi\in\kk^{\times}$, then $\varphi^{L}$ is a polynomial.
\end{conjecture}

The two conjectures are linked by two bridges. The first is immediate;
the second is the doubling construction of \cite{Meng2006} (see also
the symmetric reduction of de~Bondt--van den Essen
\cite{deBondtvandenEssen2005}, and \cite{Zhao2007} for related
formulations).

\begin{proposition}\label{prop:JCtoHC}
$\JC_n \Longrightarrow \HC_n$.
\end{proposition}

\begin{proof}[Proof (following \cite{Meng2006}).]
If $\det\Hess\varphi\in\kk^{\times}$ then $\nabla\varphi$ is a
polynomial map with constant nonzero Jacobian determinant, so by
$\JC_n$ it has a polynomial inverse $G$. By uniqueness of the formal
inverse at the base point, $G$ coincides with the formal inverse
$x(p)$ of $p=\nabla\varphi(x)$ appearing in the definition of
$\varphi^{L}$. Substituting into the defining formula
$\varphi^{L}(p)=\langle p,x(p)\rangle-\varphi(x(p))$ gives
\[
\varphi^{L}(p)=\langle p,G(p)\rangle-\varphi(G(p)),
\]
a polynomial in $p$. Hence $\varphi^{L}$ is a polynomial.
\end{proof}

\begin{proposition}\label{prop:HCtoJC}
$\HC_{2n} \Longrightarrow \JC_n$.
\end{proposition}

\begin{proof}[Proof (following \cite{Meng2006}).]
Given $F\colon\kk^n\to\kk^n$ with $\Jac F\in\kk^{\times}$, set
$\varphi(x,y)=\langle y,F(x)\rangle$, a polynomial in
$\kk[x_1,\dots,x_n,y_1,\dots,y_n]$.
Since $\varphi$ is linear in $y$, the Hessian has the block form
\begin{equation}\label{eq:block}
\Hess\varphi=
\begin{pmatrix}
\sum_k y_k\,\Hess F_k(x) & \;(DF(x))^{T}\\[2pt]
DF(x) & 0
\end{pmatrix},
\end{equation}
and swapping the two blocks of $n$ rows (a permutation of sign
$(-1)^{n^2}=(-1)^n$) gives
$\det\Hess\varphi=(-1)^n(\Jac F)^2\in\kk^{\times}$. If $\HC_{2n}$
holds, $\varphi^{L}$ is a polynomial, so $\nabla\varphi$ has a
polynomial inverse. But $\nabla\varphi$ is triangular in $(x,y)$,
\[
\nabla\varphi(x,y)=\bigl((DF(x))^{T}y,\;F(x)\bigr):
\]
the $q$-block $F(x)$ depends on $x$ alone, while for each fixed $x$ the
$p$-block $y\mapsto(DF(x))^{T}y$ is an invertible linear map, since
$\Jac F\in\kk^{\times}$. Its inverse is therefore triangular too, and
recovers $x$ from $q=F(x)$ by a polynomial map---a polynomial inverse of
$F$. Hence $\JC_n$ holds.
\end{proof}

The counterexamples below all refute $\HC_n$ the same way, through the
following elementary observation.

\begin{lemma}\label{lem:noninj}
Let $\varphi$ satisfy $\det\Hess\varphi\in\kk^{\times}$. If the gradient
map $\nabla\varphi$ is not injective, then $\varphi^{L}$ is not a
polynomial.
\end{lemma}

\begin{proof}
Were $\varphi^{L}$ a polynomial, $\nabla\varphi^{L}$ would be a
polynomial map, and the formal identity
$\nabla\varphi^{L}\circ\nabla\varphi=\mathrm{id}$, valid near the base
point, would be an identity of polynomials, hence hold on all of
$\kk^{n}$; thus $\nabla\varphi^{L}$ would be a global inverse of
$\nabla\varphi$, impossible for a non-injective map.
\end{proof}

\section{Two stabilization lemmas}

Falsehood propagates upward in dimension, on both sides.

\begin{lemma}\label{lem:JCstab}
If $\JC_n$ is false, then $\JC_{n+1}$ is false.
\end{lemma}

\begin{proof}
If $F\colon\kk^n\to\kk^n$ is a Keller map, so is
$\widetilde F(x,t)=(F(x),t)$ in dimension $n+1$, with
$\Jac\widetilde F=\Jac F\in\kk^{\times}$. As $\widetilde F$ acts on $x$
and $t$ separately, it is invertible as a polynomial map if and only if
$F$ is. Hence a counterexample to $\JC_n$ lifts to one for $\JC_{n+1}$.
\end{proof}

\begin{lemma}\label{lem:HCstab}
If $\HC_n$ is false, then $\HC_{n+1}$ is false.
\end{lemma}

\begin{proof}
If $\varphi\in\kk[x_1,\dots,x_n]$ has $\det\Hess\varphi\in\kk^{\times}$,
so does $\widetilde\varphi(x,t)=\varphi(x)+\tfrac12 t^2$ in dimension
$n+1$, since $\Hess\widetilde\varphi=\Hess\varphi\oplus(1)$. As
$\widetilde\varphi$ is a sum in the separate variables $x$ and $t$, its
Legendre transform separates likewise,
$\widetilde\varphi^{\,L}(p,s)=\varphi^{L}(p)+\tfrac12 s^2$, so
$\widetilde\varphi^{\,L}$ is a polynomial if and only if $\varphi^{L}$
is. Hence a counterexample to $\HC_n$ lifts to one for $\HC_{n+1}$.
\end{proof}

\section{The status theorem}

\begin{theorem}\label{thm:status}
Over any field of characteristic zero:
\begin{enumerate}
\item[\(\mathrm{(i)}\)] $\JC_1$ and $\HC_1$ are true; moreover $\HC_2$
and $\HC_3$ are true.
\item[\(\mathrm{(ii)}\)] $\JC_n$ is false for every $n\ge3$.
\item[\(\mathrm{(iii)}\)] $\HC_n$ is false for every $n\ge5$.
\item[\(\mathrm{(iv)}\)] Within the two indexed families $\JC_n$ and
$\HC_n$, $n\ge1$, over fields of characteristic zero, these results
leave exactly two statements undecided, $\JC_2$ and $\HC_4$, linked by
\[
\HC_4 \;\Longrightarrow\; \JC_2 .
\]
\end{enumerate}
\end{theorem}

\begin{proof}
(i) The one-variable cases are classical: $F'\in\kk^{\times}$ forces
$F$ linear, and $\varphi''\in\kk^{\times}$ forces $\varphi$ quadratic,
hence $\varphi^{L}$ quadratic. For $\HC_2$ and $\HC_3$, de~Bondt
\cite{deBondt2012} proved that a polynomial in at most three variables
with constant nonzero Hessian determinant has a polynomial gradient
inverse; the argument in the proof of Proposition~\ref{prop:JCtoHC} then
makes $\varphi^{L}$ a polynomial.

(ii) $\JC_3$ is false by \eqref{eq:Fmap}--\eqref{eq:collision};
Lemma~\ref{lem:JCstab} propagates falsehood to all $n\ge3$.

(iii) $\HC_5$ is false by the explicit five-variable counterexample of
Section~\ref{sec:psi5} (Theorem~\ref{thm:psi5}); Lemma~\ref{lem:HCstab}
propagates falsehood to all $n\ge5$. Independently, the falsity of
$\JC_3$ gives the falsity of $\HC_6$ through
Proposition~\ref{prop:HCtoJC} (contrapositive), recovering the range
$n\ge6$.

(iv) Parts (i)--(iii) decide $\JC_n$ for every $n\ne2$ and $\HC_n$ for
every $n\ne4$, leaving exactly $\JC_2$ and $\HC_4$ undecided; these are
linked by $\HC_4\Rightarrow\JC_2$ (Proposition~\ref{prop:HCtoJC} with
$n=2$).
\end{proof}

\medskip
\noindent
The resulting landscape:

\begin{center}
\renewcommand{\arraystretch}{1.3}
\begin{tabular}{c|ccccc}
$n$ & $1$ & $2$ & $3$ & $4$ & $\ge5$\\
\hline
$\JC_n$ & true & \textbf{open} & false & false & false\\
$\HC_n$ & true & true & true & \textbf{open} & false\\
\end{tabular}

\smallskip
{\small The only bridge between the two surviving open statements is
$\HC_4\Rightarrow\JC_2$.}
\end{center}

\noindent
Both survivors are open: $\JC_2$ is the classical two-dimensional
Jacobian conjecture, unproven since 1939 (see \cite{vandenEssen2000}),
and no result known to us settles $\HC_4$. They are linked only by
$\HC_4\Rightarrow\JC_2$ (Proposition~\ref{prop:HCtoJC} with $n=2$), so
$\HC_4$ is the stronger: proving it settles $\JC_2$, while a
counterexample to $\JC_2$ falsifies it. That both sit in the lowest
dimensions may point to a genuinely low-dimensional phenomenon.

\begin{remark}[Why the descent stops at five variables]
\label{rem:selfterm}
The Schur descent that produces $\HC_5$ from $\HC_6$
(Section~\ref{sec:psi5}) cannot be iterated to reach $\HC_4$: the
mechanism self-terminates. Each descent needs a pivot variable in which
the potential is affine-linear with a degenerate reduced Hessian, a
degeneracy the doubled potential owes to its linearity in the dual
variables. A single descent already spends that linearity---the
descended potential is quadratic in those variables---so no further
pivot of the required type remains. Reaching $\HC_4$ therefore appears
to require genuinely new ideas.
\end{remark}

\section{An explicit counterexample to \texorpdfstring{$\HC_6$}{HC6}}

Applying the doubling of Proposition~\ref{prop:HCtoJC} to the map
\eqref{eq:Fmap} yields a completely explicit counterexample to $\HC_6$
---the case $n=6$ of Theorem~\ref{thm:status}(iii), and hence every
$n\ge6$ by Lemma~\ref{lem:HCstab}. Write $u=1+x_1x_2$ and set
\begin{equation}\label{eq:phi}
\begin{aligned}
\varphi(x_1,x_2,x_3,y_1,y_2,y_3)
={}& y_1\bigl(u^3x_3+x_2^2u(4+3x_1x_2)\bigr)\\
&+ y_2\bigl(x_2+3x_1u^2x_3+3x_1x_2^2(4+3x_1x_2)\bigr)\\
&+ y_3\bigl(2x_1-3x_1^2x_2-x_1^3x_3\bigr),
\end{aligned}
\end{equation}
a polynomial of degree $8$.

\begin{proposition}\label{prop:HC6}
The polynomial \eqref{eq:phi} satisfies
$\det\Hess\varphi\equiv-4$, while $\nabla\varphi$ is not injective;
consequently $\varphi^{L}$ is not a polynomial and $\HC_6$ is false.
\end{proposition}

\begin{proof}
Here $\varphi=\langle y,F(x)\rangle$ is the doubling of
Proposition~\ref{prop:HCtoJC} applied to $F$, and that proof supplies
both claims. First, by the block form \eqref{eq:block},
$\det\Hess\varphi=(-1)^3(\Jac F)^2=-(-2)^2=-4$. Second,
$\nabla\varphi(x,y)=\bigl((DF(x))^{T}y,\,F(x)\bigr)$, so a collision of
$F$ lifts at once to one of $\nabla\varphi$: for the three colliding
points $P=(0,0,-\tfrac14),(1,-\tfrac32,\tfrac{13}2),
(-1,\tfrac32,\tfrac{13}2)$ of \eqref{eq:collision}, the distinct points
$(P,0)\in\kk^{6}$ share the gradient
\[
\nabla\varphi(P,0)=\bigl(0,\,F(P)\bigr)=\Bigl(0,0,0,-\tfrac14,0,0\Bigr).
\]
Hence $\nabla\varphi$ is not injective, and $\varphi^{L}$ is not a
polynomial by Lemma~\ref{lem:noninj}. (Only non-injectivity of
$\nabla\varphi$ is used; we make no claim about its generic fiber
cardinality.)
\end{proof}

\section{An explicit counterexample to \texorpdfstring{$\HC_5$}{HC5}}\label{sec:psi5}

The six-variable counterexample $\varphi$ of \eqref{eq:phi} has one variable
more than we would like. Every component of the map \eqref{eq:Fmap} is
affine-linear in $x_3$, so we may eliminate $x_3$ and obtain a
genuine five-variable counterexample. The elimination---a one-variable
\emph{Schur descent}, found by the second author---is
carried out, and shown to work in general, in Appendix~\ref{app:schur}. Here we simply write down the resulting
polynomial and check by hand the two properties that make it a
counterexample.

Collecting $\varphi=x_3A+B$ according to its dependence on $x_3$, and
writing $u=1+x_1x_2$ as before, gives the two polynomials
\begin{align*}
A&=y_1u^3+3x_1y_2u^2-x_1^3y_3,\\
B&=y_1x_2^2u(4+3x_1x_2)
+y_2\bigl(x_2+3x_1x_2^2(4+3x_1x_2)\bigr)
+y_3\bigl(2x_1-3x_1^2x_2\bigr)
\end{align*}
in the five variables $x_1,x_2,y_1,y_2,y_3$. The polynomial of the next
theorem is built from $A$ and $B$ alone; Appendix~\ref{app:schur}
explains how it arises and why this combination transports the collision.

\begin{theorem}\label{thm:psi5}
The five-variable polynomial
\[
\Psi\;=\;A^2+13A+2B\;\in\;
\mathbb{Z}[x_1,x_2,y_1,y_2,y_3]
\]
has total degree $14$ and satisfies
\[
\det\Hess\Psi\;\equiv\;128 .
\]
Moreover the two distinct rational points
\[
P_{\pm}=\bigl(\pm1,\;\mp\tfrac32,\;0,\,0,\,0\bigr)
\]
satisfy
\[
\nabla\Psi(P_+)\;=\;
\nabla\Psi(P_-)\;=\;
\bigl(0,\,0,\,-\tfrac12,\,0,\,0\bigr).
\]
Consequently $\Psi^{L}$ is not a polynomial, and $\HC_5$ is false.
\end{theorem}

\begin{proof}
\emph{Constant Hessian determinant.} The doubled potential
$\varphi=x_3A+B$ of \eqref{eq:phi} is affine-linear in $x_3$, so its
Hessian in $(x_3,w)$, with $w=(x_1,x_2,y_1,y_2,y_3)$, is the bordered
matrix $\left(\begin{smallmatrix}0&g^{T}\\ g&M\end{smallmatrix}\right)$,
where $g=\nabla_wA$ and $M=M(s,w)$ with
$M(s,w)\coloneqq\Hess_w(B+sA)$, at $s=x_3$. Since
$B+sA$ is linear in $(y_1,y_2,y_3)$, the three $y$-rows of $M$ are
supported in the two $x$-columns, so $\det M(s,w)\equiv0$; and expanding
the border gives $-g^{T}\adj(M(s,w))\,g=\det\Hess\varphi=-4$,
independently of $s$ (Proposition~\ref{prop:HC6}); here $\adj$ denotes the
adjugate (classical adjoint) matrix, characterized by
$\adj(M)\,M=(\det M)\,I$. By the rank-one
determinant identity
$\det(M+\lambda gg^{T})=\det M+\lambda\,g^{T}\adj(M)g$
(immediate on expanding in the columns, since $gg^{T}$ has rank one), applied to
$\Hess_w\bigl(B+\tfrac12A^2+\tfrac{13}2A\bigr)=M\bigl(\tfrac{13}2+A,w\bigr)+gg^{T}$,
\[
\det\Hess_w\!\Bigl(B+\tfrac12A^2+\tfrac{13}2A\Bigr)
=0+g^{T}\adj(M)g=4 .
\]
As $\Psi=2\bigl(B+\tfrac12A^2+\tfrac{13}2A\bigr)$ and scaling a
potential in five variables by $2$ multiplies its Hessian determinant by
$2^{5}$, we get $\det\Hess\Psi=2^{5}\cdot4=128$. The total degree is
$14$, and expanding $A^2+13A+2B$ yields $42$ monomials (counted after
expansion; the count is confirmed in Appendix~\ref{app:computations}).

\emph{Gradient collision.} Write $F=F_0+x_3F_1$, so $A=\langle y,F_1\rangle$
and $B=\langle y,F_0\rangle$. Differentiating gives, everywhere,
$\nabla\Psi=(2A+13)\nabla A+2\nabla B$. Since $A$ and $B$ are
homogeneous of degree one in $(y_1,y_2,y_3)$, so are their
$x$-derivatives; hence $A$, $B$, and their $x$-derivatives all vanish on
$\{y=0\}$. The $x$-block of $\nabla\Psi$ therefore vanishes on
$\{y=0\}$, while---using
$A=0$ and $\nabla_yA=F_1$, $\nabla_yB=F_0$---its $y$-block is
$13F_1+2F_0=2\bigl(F_0+\tfrac{13}2F_1\bigr)
=2F(\,\cdot\,,\tfrac{13}2)$. Thus
\[
\nabla\Psi(P_\pm)
=\bigl(0,0,\,2F(\pm1,\mp\tfrac32,\tfrac{13}2)\bigr)
=(0,0,-\tfrac12,0,0),
\]
the two values agreeing because
$F(1,-\tfrac32,\tfrac{13}2)=F(-1,\tfrac32,\tfrac{13}2)=(-\tfrac14,0,0)$
by \eqref{eq:collision}, while $P_+\ne P_-$.

\emph{Conclusion.} The distinct points
$P_+\ne P_-$ have equal gradient, so
$\nabla\Psi$ is not injective; since
$\det\Hess\Psi\in\kk^{\times}$, Lemma~\ref{lem:noninj} shows
$\Psi^{L}$ is not a polynomial and $\HC_5$ is false. (As in
Proposition~\ref{prop:HC6}, only non-injectivity of
$\nabla\Psi$ is used; we make no claim about the generic fiber
cardinality.)
\end{proof}

Because the theorem refers to $\Psi$ alone, the refutation of
$\HC_5$ is self-contained and elementary to verify by hand: the proof
reduces both properties of $\Psi$ to two facts about Alp\"oge's
map $F$---that $\Jac F$ is a nonzero constant, and that $F$ maps
$(\pm1,\mp\tfrac32,\tfrac{13}2)$ to a common image. The first is a single
$3\times3$ determinant (a conceptual explanation of which now also appears
in the discussion documented in \cite{Alpoge2026}), the second an
immediate substitution; no computer algebra is needed.

Alp\"oge's map has a three-point fiber \eqref{eq:collision}, yet
$\Psi$ exhibits only a two-point gradient collision. This is not
an oversight: the descent transfers exactly those collisions whose source
points share the pivot coordinate $x_3$
(Lemma~\ref{lem:collision-transfer}). Only the pair
$(\pm1,\mp\tfrac32,\tfrac{13}2)$ lies over a common $x_3=t_0=\tfrac{13}2$
and so is carried down to $P_\pm$; the third preimage
$(0,0,-\tfrac14)$ sits alone over $x_3=-\tfrac14$, with no partner to
collide with after the pivot is eliminated, and is simply not seen by the
five-variable potential. Two colliding points already refute $\HC_5$. The descent is treated in
general in Appendix~\ref{app:schur}, where it turns \emph{any} collision
of a Keller map over a common last coordinate into a constant-Hessian
counterexample, and where the coefficients $13$ and $2$ are seen to be
one choice among a two-parameter family.

\section*{Acknowledgments}
\addcontentsline{toc}{section}{Acknowledgments}
\noindent
The first author thanks Jianshu Li, Qiang Du, and Yu Zhao for informing
him of Alp\"oge's counterexample and for helpful discussions. The Schur
descent of Section~\ref{sec:psi5} was found by the second author; large
language model (LLM) assistants (OpenAI's ChatGPT Pro and Anthropic's
Claude) served as auxiliary tools---ChatGPT Pro in the search for the
descent, both in scripting the exact-arithmetic verifications. All
statements, computations, and proofs are the sole responsibility of the
authors, who have verified them by hand. The five-variable
counterexample is, to our knowledge, new.

\appendix
\section{The Schur-descent bridge in general odd dimensions}
\label{app:schur}

The five-variable counterexample of Section~\ref{sec:psi5} was produced from the
six-variable doubling by a mechanism that is not special to the displayed
coefficients. We record here the two elementary lemmas underlying the
construction and the general statement they yield.

Throughout, $\kk$ is a field of characteristic zero. Suppose a potential
in the scalar variable $t$ and the block $w=(w_1,\dots,w_m)$ is
affine-linear in $t$,
\begin{equation}
\Phi(t,w)=t\,A(w)+B(w).
\label{eq:app-affine-t}
\end{equation}
Write $g=\nabla_w A$ and, for an auxiliary indeterminate $s$,
\[
M(s,w)\coloneqq\Hess_w B+s\,\Hess_w A,
\]
a matrix affine-linear in $s$ whose entries are polynomials in $s$ and
$w$; we may therefore substitute for $s$ any polynomial expression---below
the variable $t$, and later a function of $w$. With $s=t$ the lower-right
block is $\Hess_w\Phi=\Hess_w(B+tA)=M(t,w)$, so the full Hessian is the
bordered matrix
\begin{equation}
\Hess_{t,w}\Phi=
\begin{pmatrix}
0 & g^{T}\\
g & M(t,w)
\end{pmatrix}.
\label{eq:app-bordered}
\end{equation}

\begin{lemma}[Schur descent]
\label{lem:schur}
Assume that in the situation \eqref{eq:app-affine-t}--\eqref{eq:app-bordered},
\[
\det\Hess_{t,w}\Phi=c\in\kk^{\times},
\qquad
\det M(s,w)\equiv0 ,
\]
the latter identically in $s$ and $w$.
For any $\lambda\in\kk^{\times}$ and $\mu\in\kk$ define the
$m$-variable polynomial
\begin{equation}
\psi_{\lambda,\mu}(w)=B(w)+\tfrac{\lambda}{2}A(w)^{2}+\mu A(w).
\label{eq:app-psi}
\end{equation}
Then
\begin{equation}
\det\Hess\psi_{\lambda,\mu}=-\lambda c .
\label{eq:app-descent}
\end{equation}
\end{lemma}

\begin{proof}
The polynomial $\psi_{\lambda,\mu}$ is a one-variable partial Legendre
transform of $\Phi$ in $t$. Repair $\Phi$ by a quadratic,
\[
\widehat\Phi(t,w)=\Phi(t,w)-\tfrac{1}{2\lambda}(t-\mu)^{2},
\]
so that $\partial_t\widehat\Phi=A(w)-\tfrac1\lambda(t-\mu)$ vanishes
exactly at $t^{*}=\mu+\lambda A(w)$, where $\widehat\Phi$ attains the
critical value $\psi_{\lambda,\mu}(w)$. The pivot
$\partial_t^2\widehat\Phi=-1/\lambda$ is invertible, so eliminating $t$
leaves on the remaining variables the \emph{Schur complement} of that
pivot in $\Hess_{t,w}\widehat\Phi$---namely the Hessian
$M(\mu+\lambda A,w)+\lambda gg^{T}$ of the critical value
$\psi_{\lambda,\mu}$---and the block-determinant identity reads
\[
\det\Hess_{t,w}\widehat\Phi
=-\tfrac1\lambda\,\det\Hess\psi_{\lambda,\mu}.
\]
It remains to compute the left-hand side. Because $\Phi$ is affine in
$t$, the repair alters only the corner of the bordered Hessian
\eqref{eq:app-bordered}, replacing $0$ by $-1/\lambda$. By the Cauchy
bordered-determinant expansion
\begin{equation}
\det\begin{pmatrix} a & g^{T}\\ g & M\end{pmatrix}
= a\,\det M - g^{T}\adj(M)\,g
\label{eq:app-cauchy}
\end{equation}
(valid for any scalar $a$; expand the border by multilinearity along its
first row and column), a bordered determinant with $\det M\equiv0$
equals $-g^{T}\adj(M)\,g$, \emph{independently of the corner $a$}. Hence
the repair leaves the determinant unchanged,
$\det\Hess_{t,w}\widehat\Phi=\det\Hess_{t,w}\Phi=c$, and the displayed
Schur identity gives $\det\Hess\psi_{\lambda,\mu}=-\lambda c$.
\end{proof}

\begin{lemma}[Transfer of a collision]
\label{lem:collision-transfer}
Suppose two distinct points $(t_0,w_+)\ne(t_0,w_-)$ sharing the same
$t$-coordinate satisfy
\[
\nabla\Phi(t_0,w_+)=\nabla\Phi(t_0,w_-),
\qquad
A(w_+)=A(w_-)=A_0 .
\]
(The second condition is automatic: $\partial_t\Phi=A(w)$ is a component
of $\nabla\Phi$, so equal full gradients at $(t_0,w_\pm)$ already force
$A(w_+)=A(w_-)$; we record it separately only to name the common value
$A_0$.) Choosing $\mu=t_0-\lambda A_0$, one has
\[
\nabla\psi_{\lambda,\mu}(w_+)=\nabla\psi_{\lambda,\mu}(w_-).
\]
\end{lemma}

\begin{proof}
At either point, using $\mu+\lambda A=\mu+\lambda A_0=t_0$,
\[
\nabla_w\psi_{\lambda,\mu}
=\nabla_w B+(\mu+\lambda A)\,\nabla_w A
=\nabla_w B+t_0\,\nabla_w A
=\nabla_w\Phi(t_0,w),
\]
and the right-hand sides agree at $w_+$ and $w_-$ by hypothesis.
\end{proof}

Together the two lemmas turn any suitable doubled Keller map into a
constant-Hessian polynomial in one fewer variable with a surviving
gradient collision.

\begin{proposition}[Odd-dimensional Schur-descent bridge]
\label{prop:bridge}
Let $F\colon\kk^{n}\to\kk^{n}$ be a Keller map, affine-linear in the
source coordinate $t=x_n$---that is, each component $F_i$ has total
degree at most $1$ in $x_n$ (no coordinate change is applied)---and let
$\Phi(x,y)=\langle y,F(x)\rangle$ be its doubled potential. Suppose $F$ has a collision over a common last
coordinate:
\[
x'_+\ne x'_-,\qquad F(x'_+,t_0)=F(x'_-,t_0),
\]
where $x'=(x_1,\dots,x_{n-1})$ and $t_0\in\kk$. Then for every
$\lambda\in\kk^{\times}$ the $(2n-1)$-variable polynomial
\[
\psi_{\lambda}(w)=B(w)+\tfrac{\lambda}{2}A(w)^{2}+t_0\,A(w)
\]
satisfies $\det\Hess\psi_{\lambda}=(-1)^{n+1}\lambda(\Jac F)^{2}
\in\kk^{\times}$ and has a gradient collision at the two distinct points
$(x'_\pm,0)$. In particular $\nabla\psi_{\lambda}$ is non-injective and
$\psi_{\lambda}^{L}$ is not a polynomial.
\end{proposition}

\begin{proof}
Write $\Phi=tA+B$ as in \eqref{eq:app-affine-t} with
$w=(x_1,\dots,x_{n-1},y_1,\dots,y_n)$; if $F=F_0+tF_1$ then
$A=\langle y,F_1\rangle$ and $B=\langle y,F_0\rangle$. Since $F$ is a
Keller map, the doubling has $\det\Hess_{t,w}\Phi=(-1)^{n}(\Jac F)^{2}
=:c\in\kk^{\times}$ by the block form \eqref{eq:block}
(Proposition~\ref{prop:HCtoJC}). For every $s$ the polynomial $B+sA$ is
linear in the $n$ dual variables $y_1,\dots,y_n$, so in
$M(s,w)=\Hess_w(B+sA)$ the $n$ rows indexed by $y$ are supported entirely
in the $n-1$ columns indexed by $x'$, hence linearly dependent, so
$\det M(s,w)\equiv0$.
Both hypotheses of Lemma~\ref{lem:schur} thus hold, and with $\mu=t_0$ it
gives $\det\Hess\psi_{\lambda}=-\lambda c=(-1)^{n+1}\lambda(\Jac F)^{2}$, a
nonzero constant.

For the collision, take the two points $(x'_\pm,0)$. At $y=0$ one has
$A=\langle0,F_1\rangle=0$ and
$\nabla\Phi(t_0,(x',0))=\bigl(0,\,F(x',t_0)\bigr)$, so the equal-$t$
collision $F(x'_+,t_0)=F(x'_-,t_0)$ gives
$\nabla\Phi(t_0,(x'_+,0))=\nabla\Phi(t_0,(x'_-,0))$ with common value
$A_0=0$. Lemma~\ref{lem:collision-transfer}, with $\mu=t_0$, transfers
this to $\nabla\psi_{\lambda}(x'_+,0)=\nabla\psi_{\lambda}(x'_-,0)$; as
$(x'_+,0)\ne(x'_-,0)$, $\nabla\psi_{\lambda}$ is not injective, and
$\psi_{\lambda}^{L}$ is not a polynomial by Lemma~\ref{lem:noninj}.
\end{proof}

\begin{remark}
Only a collision of $F$ over a common last coordinate is needed---not a
collision of the doubled gradient $\nabla\Phi$, which is produced
automatically at $y=0$. The construction is a genuine two-parameter
family: $\lambda\in\kk^{\times}$ scales the determinant, while
Lemma~\ref{lem:collision-transfer} allows any collision lying over a value
$A=A_0$ to be transported by taking $\mu=t_0-\lambda A_0$ in
$\psi_{\lambda,\mu}=B+\tfrac{\lambda}{2}A^2+\mu A$; the choice $A_0=0$
above is the simplest. When $F$, the points $x'_\pm$, and $t_0,\lambda$
are defined over $\Q$, a scalar multiple of $\psi_\lambda$ has integral
coefficients. Over a general field no such integrality holds, and none is
claimed.
\end{remark}

For $n=3$, Proposition~\ref{prop:bridge} recovers the counterexample of
Theorem~\ref{thm:psi5}: the map $F$ of \eqref{eq:Fmap} is affine-linear
in $t=x_3$ and, by \eqref{eq:collision}, its points $(1,-\tfrac32)$ and
$(-1,\tfrac32)$ collide over $t_0=\tfrac{13}2$, so with $\lambda=1$ and
$\mu=t_0=\tfrac{13}2$ it produces $\psi=B+\tfrac12A^2+\tfrac{13}2A$, whose
double $\Psi=2\psi=A^2+13A+2B$ is the polynomial of that
theorem. So $\Psi$ is neither unique nor canonical: the second
parameter $\lambda$ of the family above only rescales the determinant, so
fixing $\det\Hess=128$ (that is, $\lambda=1$) still leaves $\mu$ free---%
every $A^2+2\mu A+2B$ is a counterexample---and $\mu=\tfrac{13}2$ is
singled out only as the value placing the collision at the trivial $y=0$
points $P_\pm$. It is the
\emph{quadratic} dependence on the dual variables $y_1,y_2,y_3$, through
the term $A^2$, that supplies the nondegeneracy a merely linear doubling
such as \eqref{eq:phi} cannot.

\section{Verification and ancillary files}
\label{app:computations}

The proof of Theorem~\ref{thm:psi5} is complete and hand-checkable; all
its identities have also been verified independently in exact rational
arithmetic. A fully expanded, line-by-line transcript---including the
$3\times3$ reduction of $\Jac F$, the block Hessian, and the six-variable
algebraic Legendre transform of $\varphi$ showing that $\varphi^{L}$ is
not a polynomial---is provided as the ancillary document
\texttt{derivation\_full\_details.pdf}. The identity
$\det\Hess\Psi=128$ is proved symbolically, for the whole
family $A^2+2\mu A+2B$, in \texttt{hc5\_schur\_descent.py} and
\texttt{independent\_check.py}, and for the representative
$\Psi=A^2+13A+2B$ (its $y=0$ collision, degree $14$, and $42$
monomials) in \texttt{coefficient\_analysis.py} and, using no external
CAS, in \texttt{independent\_nosympy.py}. All scripts are fail-closed
(Python~3.8, \textsc{SymPy}~1.12; see \texttt{verification/README.md})
and, with the ancillary document, are available for independent
re-running in the public GitHub repository
\url{https://github.com/malyang/hc5-counterexample}.%
\footnote{The repository URL will be finalized at acceptance and, if a
citable snapshot is desired, an archival release can additionally be
minted a DOI.}


\end{document}